\documentclass{amsart}
\usepackage[T1]{fontenc}
\usepackage{amsmath}
\usepackage{amssymb}
\usepackage{tikz-cd}

\usepackage[style=alphabetic,backend=biber]{biblatex}
\newtheorem{definition}[subsection]{Definition}
\newtheorem{lemma}[subsection]{Lemma}

\newcommand{\chara}{{\textup{char}}}
\newcommand{\Spec}{{\textup{Spec}}}
\newcommand{\Z}{{\mathbb{Z}}}
\renewcommand{\P}{{\mathbb{P}}}
\newcommand{\Q}{{\mathbb{Q}}}

\title{All subterminal schemes}
\author{Jaime Benabent Guerrero}
\date{\today}

\begin{document}
	
	\begin{abstract}
        We classify all subterminal schemes by characterizing their point structure, stalks, and topologies.
        This extends our previous classification of subterminal affine schemes, which correspond to spectra of solid rings.
	\end{abstract}
	
	\maketitle
	
	\section{Introduction}
    \noindent

    An object in a category is \emph{solid} if it admits at most one morphism into any other object. 
    In the category of commutative rings, solid rings were first studied in \cite{BK71}, and a complete classification was given in \cite{BG25}.
    When translated to the opposite category, the category of affine schemes, this corresponds precisely to subterminal objects, those objects that admit at most one morphism from any other object.

    \begin{definition}
        An object $X$ in a category is subterminal if, for any object $A$, there is at most one morphism $A \to X$.
    \end{definition}

    Thus, the classification of solid rings immediately yields a classification of subterminal objects among affine schemes: they are precisely the spectra of solid rings.
    The natural next question is whether additional subterminal objects exist in the broader category of all schemes.
    The goal of this paper is to answer this question by providing a complete classification of subterminal schemes.

    Our approach is as follows:
    \begin{itemize}
        \item
        We first recall the classification of solid rings and show that their spectra remain subterminal within the category of schemes.
        
        \item
        We then analyze the structure of these spectra, focusing on their topological properties and local rings.

        \item 
        Using this understanding, we determine whether other subterminal schemes exist and classify all possible such schemes.  
        To do so, we first analyze their points and stalks, then construct a compatible topology by covering them with spectra of solid rings. 
    \end{itemize}
    
    To recall the classification of solid rings, we use the tower of inclusions perspective introduced in \cite{BG25} for solid rings of the fourth type.
    The four types of solid rings are:

    \begin{enumerate}
		\item Cyclic rings $\Z/n\Z$, where $n \geq 1$ is a positive integer.
			
		\item Subrings of the rationals $\Z[J^{-1}]$, where $J$ is any set of primes.
			
		\item Product rings $\Z[J^{-1}] \times \Z/n\Z$, where $J$ is any set of primes that contains all prime factors of $n \geq 2$.

        \item Colimits of diagrams of the form
        \[
        \begin{tikzcd}
        {\Z[J_0^{-1}]} \arrow[r, hook] & \cdots \arrow[r, hook] & {\Z[J_n^{-1}] \times \prod\limits_{p \in K_n} \Z/p^{e_p}\Z} \arrow[r, hook] & \cdots,
        \end{tikzcd}
        \]
        where $J$ is an infinite collection of primes, and $K := \{p_1, p_2, p_3, \dots\}$ is an infinite subset of $J$, so that we may define $K_n := \{p_i\}_{i = 1}^n$ and $J_n := (J \setminus K) \cup K_n$, and $\{e_p\}_{p \in K}$ is a family of positive exponents.
	\end{enumerate}

    Understanding the structure of their spectra is central to our study.
    In the final section, we show that any subterminal scheme must be built from the spectra of solid rings.
    This allows us to establish necessary conditions on the points and local rings of subterminal schemes, analyze their topology, and determine their possible global structures.
    Our results provide a complete classification that characterizes all subterminal schemes.

    \section{Spectra of solid rings}
    \noindent

    In this section, we establish that affine schemes arising from solid rings remain subterminal within the broader category of schemes.
    Additionally, we examine the structure of their spectra.\\

    We begin with a general lemma that we then specialize to affine schemes.

    \begin{lemma}
    Let $S$ be a scheme such that for every affine scheme $\Spec(A)$, there exists at most one morphism $\Spec(A) \to S$.
    Then $S$ is a subterminal object in the category of schemes.
    \end{lemma}
    \begin{proof}
    To prove that $S$ is subterminal, we must show that for every scheme $X$, there exists at most one morphism $X \to S$.

    Given any scheme $X$, consider an open affine cover $\{\Spec(A_i)\}_{i \in I}$ of $X$.
    By assumption, for each $i$, there is at most one morphism $\Spec(A_i) \to S$.
    If such morphisms exist, they must agree on overlaps, since morphisms of schemes are determined by their restrictions to open affine subschemes.
    This ensures that the local morphisms patch together uniquely to define at most one morphism $X \to S$.

    Thus, $S$ is subterminal in the category of schemes.
    \end{proof}

    \begin{lemma}
        An affine scheme is subterminal if and only if it is the spectrum of a solid ring.
    \end{lemma}
    \begin{proof}
        Since the category of affine schemes fully embeds into the category of schemes, any morphism between affine schemes is determined within the category of affine schemes.
        It follows that the only subterminal affine schemes are those corresponding to solid rings.
    \end{proof}

    We now analyze the structure of the spectrum of a solid ring.

    \begin{lemma}
    The localization of a solid ring at a multiplicative subset is solid.
    \end{lemma}
    \begin{proof}
    Let $R$ be a solid ring, let $S \subseteq R$ be a multiplicative subset and let $R_S := S^{-1}R$, the localization of $R$ away from $S$.
    
    A morphism from $R_S \to A$ is equivalent to a morphism $R \to A$ sending the elements of $S$ to invertible elements of $A$.
    Thus, since there is at most one morphism $R \to A$ there is at most one morphism $R_S \to A$.    
    \end{proof}

    Since we have a full classification of solid rings, we can explicitly describe their prime ideals and their localizations.
    To conclude, we provide a detailed analysis of each type of solid ring, focusing on their topological properties and stalks at each point.

    \subsection{Solid rings of type (1)}
    For a cyclic ring $\Z/n\Z$, where $n \geq 1$ is a positive integer, we can factor $n$ into its prime components as
    $$
    \Z/n\Z \cong \prod\limits_{p \mid n} \Z/p^{v_p(n)}\Z.
    $$
    It follows immediately that $\Spec(\Z/n\Z)$ is a finite discrete space, with one point for each prime divisor $p$ of $n$.
    Moreover, the stalk at the point associated with $p$ is the ring $\Z/p^{v_p(n)}\Z$.

    \[    
    \begin{tikzpicture}
    \node[label=below:$(p_1)$] at (1,0) {\textbullet};
    \node[label=below:$(p_2)$] at (2,0) {\textbullet};
    \node at (3,0) {$\cdots$};
    \node[label=below:$(p_k)$] at (4,0) {\textbullet};
    \end{tikzpicture}
    \]

    \subsection{Solid rings of type (2)}
    For a subring of the rationals $\Z[J^{-1}]$, where $J$ is any set of primes, the spectrum consists of one point for each prime $p \not\in J$, along with a generic point corresponding to the zero ideal, which can be represented as a line.
    The topology is cofinite, meaning that every nonempty open set contains all but finitely many points. Additionally, every nonempty open set contains the generic point associated with $0$.
    The stalk at the point corresponding to $p$ is the localization $\Z_{(p)}$, while the stalk at the generic point is the field of fractions $\Q$.

    \[    
    \begin{tikzpicture}
    \draw (0,0) -- (5,0);
    \node[label=below:$(p_1)$] at (1,0) {\textbullet};
    \node[label=below:$(p_2)$] at (2,0) {\textbullet};
    \node[label=below:$(p_3)$] at (3,0) {\textbullet};
    \node[label=below:$\cdots$] at (4,0) {\textbullet};
    \end{tikzpicture}
    \]

    \subsection{Solid rings of type (3)}
    A product ring $\Z[J^{-1}] \times \Z/n\Z$, where $J$ is a set of primes containing all prime factors of $n \geq 2$, is the product of a solid ring of type (1) with a solid ring of type (2).
    Consequently, its spectrum is the disjoint union of the spectra of its factors, fully determining the associated scheme.

    Notice how the points corresponding to primes dividing $n = \prod\limits_{i = 1}^l q_i^{v_{q_i}(n)}$, since they must belong to $J$, do not appear on the line in the pictorial representation, i.e., each prime appears at most once in the drawing.

    \[    
    \begin{tikzpicture}
    \draw (0,0) -- (8,0);
    \node[label=below:$(p_1)$] at (1,0) {\textbullet};
    \node[label=above:$(q_1)$] at (2,1) {\textbullet};
    \node[label=below:$\cdots$] at (3,0) {\textbullet};
    \node[label=above:$\cdots$] at (4,1) {\textbullet};
    \node[label=below:$(p_k)$] at (5,0) {\textbullet};
    \node[label=above:$(q_l)$] at (6,1) {\textbullet};
    \node[label=below:$\cdots$] at (7,0) {\textbullet};
    \end{tikzpicture}
    \]

    \subsection{Solid rings of type (4)}
    Consider colimits of diagrams of the form
    \[
    \begin{tikzcd}
    {\Z[J_0^{-1}]} \arrow[r, hook] & \cdots \arrow[r, hook] & {\Z[J_n^{-1}] \times \prod\limits_{q \in K_n} \Z/q^{e_q}\Z} \arrow[r, hook] & \cdots,
    \end{tikzcd}
    \]
    where $J$ is an infinite set of primes, and $K := \{q_1, q_2, q_3, \dots\}$ is an infinite subset of $J$.
    We define $K_n := \{q_i\}_{i = 1}^n$ and $J_n := (J \setminus K) \cup K_n$, with $\{e_q\}_{q \in K}$ denoting an indexed family of positive exponents.
    The affine spectrum of the colimit ring can be determined by taking the corresponding limit in the opposite category of affine schemes.

    
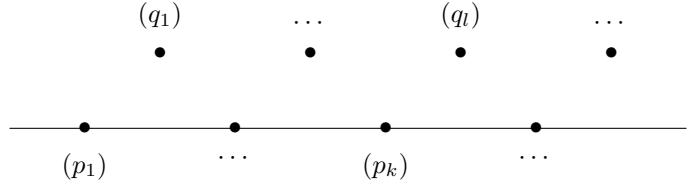
\[    
    \begin{tikzpicture}
    \draw (0,0) -- (9,0);
    \node[label=below:$(p_1)$] at (1,0) {\textbullet};
    \node[label=above:$(q_1)$] at (2,1) {\textbullet};
    \node[label=below:$\cdots$] at (3,0) {\textbullet};
    \node[label=above:$\cdots$] at (4,1) {\textbullet};
    \node[label=below:$(p_k)$] at (5,0) {\textbullet};
    \node[label=above:$(q_l)$] at (6,1) {\textbullet};
    \node[label=below:$\cdots$] at (7,0) {\textbullet};
    \node[label=above:$\cdots$] at (8,1) {\textbullet};
    \end{tikzpicture}
    \]

    The topology is given as follows:
    \begin{enumerate}
        \item Open sets containing the generic point $0$ are cofinite.
        \item Open sets that do not contain $0$ are arbitrary subsets of points associated with primes in $K$.
    \end{enumerate}

    The stalks at the points are as follows:
    \begin{enumerate}
        \item For $p \not\in J$, the stalk is $\Z_{(p)}$.
        \item For $q \in K$, the stalk is $\Z/q^{e_q}\Z$.
        \item At the generic point $0$, the stalk is the field of fractions $\Q$.
    \end{enumerate}

    \section{All subterminal schemes}
    \noindent

    In the previous section, we established that the only subterminal affine schemes are the spectra of solid rings.
    The natural question that follows is whether there exist additional subterminal schemes beyond these.
    In this section, we classify all subterminal schemes.

    We proceed as follows:
    \begin{enumerate}
        \item We determine necessary conditions on the set of points and their respective stalks for a scheme to be subterminal.

        \item We analyze the possible topologies that can be imposed on schemes satisfying these point and stalk conditions established in the step before.

        \item We provide a complete classification of subterminal schemes in a unified framework.
    \end{enumerate}
   
    \subsection{Point structure and stalks}
    
    We begin by identifying the constraints on the set of points and their stalks for a scheme to be subterminal.
    A key observation is that every open subscheme of a subterminal scheme must itself be subterminal.
    Applying this to open affine subschemes and recalling that every affine subterminal scheme is the spectrum of a solid ring, as established in the previous section, we obtain the following result:
    
    \begin{lemma}
        Every open affine subscheme of a subterminal scheme is the spectrum of a solid ring.
    \end{lemma}

    Since stalks are determined locally, this lemma allows us to characterize the possible stalks at points of a subterminal scheme.
    As established in the previous section, the stalk at any point in the spectrum of a solid ring is one of the following: $\Q$, $\Z_{(p)}$ for some prime $p$, or $\Z/p^n\Z$ for some prime $p$ and positive integer $n$.
    These remain the only possible stalks in a general subterminal scheme.
    
    Moreover, we can make two additional observations:

    \begin{enumerate}
        \item
        For any prime $p$, there can be at most one point whose stalk belongs to the set $\{\Z_{(p)}\} \cup \{\Z/p^n\Z\}_{n \in \Z_+}$. Otherwise, there would exist distinct inclusions from the single-point scheme $\Spec(\Z/p\Z)$, induced by the canonical morphisms $\Z_{(p)} \to \Z/p\Z$ or $\Z/p^n\Z \to \Z/p\Z$, violating the uniqueness condition required for subterminal schemes.

        \item 
        By an analogous reasoning, there can be at most one point whose stalk is $\Q$.
        Furthermore, if a scheme contains a point whose stalk is $\Z_{(p)}$ for some prime $p$, then it must also contain a point with stalk $\Q$.
    \end{enumerate}

    Thus, as a set of points and their respective stalks, we have completely determined the possible configurations of points in a subterminal scheme.
    Even more, given a scheme with such configuration of points and their respective stalks, the scheme is subterminal:
    
    \begin{lemma}
        Let $R$ be a scheme satisfying the conditions established above.
        Then $R$ is subterminal.
    \end{lemma}
    \begin{proof}
        Let $S$ be another scheme.
        We show that there can be at most one morphism $S \to R$.\\
        
        At the level of points, there is at most one possible choice for each point $a \in S$, whose stalk is the local ring $A$.
        We now verify that two distinct choices are impossible:
        
        First, consider the family of local rings $\{\Q, \Z_{(2)}, \Z_{(3)}, \dots\}$.
        Any morphism $\Z \to A$ must send the preimage of the maximal ideal $\mathfrak{m} \subseteq A$ to a prime ideal $\mathfrak{p} \subseteq \Z$.
        This forces $A$ to be the localization of $\Z$ at $\mathfrak{p}$, meaning the only possible local morphism from this family is from $\Z_\mathfrak{p}$.
        
        Second, at most one prime $p$ allows a morphism $\Z/p^n\Z \to A$, since such a morphism would imply that $\chara(A)$ is a power of $p$.

        Combining these observations, any point of $S$ mapping to $R$ is uniquely determined: if morphisms from both families exist, they must correspond to the same prime, since $\Z \to A$ would factor through $\Z/p^n\Z$, forcing the prime ideal corresponding to $\mathfrak{m}$ to be $\mathfrak{p} = (p)$.
        Since there is at most one point for each prime, we conclude that the choice of point in $S$ mapping to $R$ is uniquely determined.\\
        
        Finally, since morphisms of schemes are uniquely determined at the level of stalks, and the stalks of $R$ are local rings under the given conditions, there is at most one induced morphism at the level of stalks for each point.
        Thus, there is at most one morphism $S \to R$, proving that $R$ is subterminal.
    \end{proof}

    Hence, the remaining question is to determine the possible topologies that such schemes may admit, which we address in the following subsection.
    
    \subsection{Possible topologies}
    
    Having determined the conditions for the set of points and their stalks, we now examine the possible topologies that such a scheme can carry.
    Interestingly, many of these structures allow multiple valid topologies, meaning different ways of defining the open sets while still forming a scheme.

    A key fact that constrains the possible topologies is that our scheme has an affine open cover by spectra of solid rings.
    This means that every point belongs to at least one affine open subset corresponding to some solid ring.
    Since the topology of a scheme is determined locally by its affine opens, understanding how open sets behave within each affine patch helps us determine the global topology.

    Throughout this discussion, we will refer to a point in our scheme either by the prime corresponding to it or by the local ring associated with it.

    \subsubsection{Open sets excluding $\Q$}
    We first consider open sets that do not contain the generic point $\Q$.
    A fundamental observation is that if an open set excludes $\Q$, it must also exclude any point associated with a ring of the form $\Z_{(p)}$.
    This follows from the fact that, in every affine open subset, which is the spectrum of a solid ring, any open set that contains $\Z_{(p)}$ for some prime $p$ must also contain $\Q$.
    Thus, if an open set does not contain $\Q$, it must avoid all such points.

    The only remaining points are those associated with cyclic rings of the form $\Z/p^n\Z$, where $p$ is a prime and $n$ is a positive integer.
    Moreover, as seen in the previous section, in the spectrum of any solid ring, each point corresponding to $\Z/p^n\Z$ is always open. 

    As a consequence, the open sets that exclude $\Q$ must be arbitrary collections of points corresponding to rings of the form $\Z/p^n\Z$. This means that, within our scheme, the open sets not containing $\Q$ can be freely chosen among these points without additional constraints.

    \subsubsection{Open sets containing $\Q$}
    Now, let us consider the structure of open sets that contain $\Q$.
    The key is that if we take two affine open sets containing $\Q$, then their symmetric difference must be finite.
    Using this fact we are able to express our subterminal scheme as a colimit of inclusions of the spectra of solid rings, where at each step we add a single point, providing our underlying set with a topology.

    \begin{lemma}
        Let $S$ be a subterminal scheme.
        Let $\Spec(A) \hookrightarrow S$ and $\Spec(B) \hookrightarrow S$ be two affine open subsets corresponding to solid rings of type (2), (3), or (4), i.e., affine open subsets that contain $\Q$.
        Then the symmetric difference
        $$
        \Spec(A) \triangle \Spec(B) := (\Spec(A) \setminus \Spec(B)) \cup (\Spec(B) \setminus \Spec(A))
        $$
        is finite.
    \end{lemma}
    \begin{proof}
    Since $\Spec(A)$ and $\Spec(B)$ correspond to solid rings of type (2), (3), or (4), their topologies satisfy the property that any open set containing $\Q$ has a cofinite complement, that is, only finitely many points in the affine scheme are missing.

    Now, consider the intersection $\Spec(A) \cap \Spec(B)$.
    This set is open in both $\Spec(A)$ and $\Spec(B)$, so its complement within each must be finite.
    That is, both $\Spec(A) \setminus \Spec(B)$ and $\Spec(B) \setminus \Spec(A)$ are finite.

    Since the symmetric difference is precisely the union of these two finite sets, it follows that $\Spec(A) \triangle \Spec(B)$ is finite, as required.
    \end{proof}

    Using the previous lemma, we now establish a procedure to cover our scheme with a nested sequence of affine open subschemes, each included in the next.
    This construction allows us to express our scheme as a colimit of inclusions of spectra of solid rings.

    We start by taking an open affine cover of our scheme.
    Since the set of points in our underlying topological space is countable, we can refine this cover to a countable one.
    We will primarily focus on the open affines that contain $\Q$, i.e. those corresponding to spectra of solid rings of type (2), (3), and (4).

    By the previous lemma, the symmetric difference between any two affine open subsets containing $\Q$ is finite.
    This property allows us to glue two such open sets into a larger affine open set that contains both.
    Explicitly, if we have two such affine open sets
    $$
    \Spec(A) := \Spec\left(\Z[I^{-1}] \times \prod\limits_{p \in K} \Z/p^{e_p}\Z\right)
    $$
    and
    $$
    \Spec(B) := \Spec\left(\Z[J^{-1}] \times \prod\limits_{p \in L} \Z/p^{e_p}\Z\right),
    $$
    then their union can be covered by
    $$
    \Spec\left(\Z[(I \cap J)^{-1}] \times \prod\limits_{p \in (K \cup L)} \Z/p^{e_p}\Z\right).
    $$
    This new affine scheme incorporates to $\Spec(A)$ all points corresponding to $\Z_{(p)}$ for $p \in I \cap (\P \setminus J)$ and all points corresponding to $\Z/p^{e_p}\Z$ for $p \in L \setminus K$.
    Since the symmetric difference is finite, we can refine our cover further by adding just one point at a time, forming a finite sequence of inclusions in between
    $$
    \Spec(A) \hookrightarrow \Spec(A_1) \hookrightarrow \cdots \hookrightarrow \Spec(A_n).
    $$

    We can do this process step by step with all the the open affines that contain $\Q$, and, furthermore, we interleave this process with the addition of points that are covered only by spectra of solid rings of type (1), ensuring that the entire space is covered.

    After all this process takes place we are left we obtain a countable sequence of spectra of solid rings, each contained in the next, where each step increases the set of points by one.
    This nested structure allows us to express our subterminal scheme as a colimit of affine schemes.
    The resulting scheme has a topology in which open sets containing $\Q$ are cofinite for a certain initial set of points, meaning that open sets must contain all but finitely many points of the initial collection of points.

    This process can always be carried, giving us a corresponding initial collection of points for each arrangement of points and their stalks as in the previous subsection.
    Furthermore, any data of such an initial collection of points of our arrangement of points and their stalks gives an open cover of our data by spectra of solid rings, i.e., a scheme.
    
    Finally, it is straightforward that any two choices of such an initial collection of points give equivalent topological spaces if and only if their symmetric difference is finite.
    
    We have completed the classification of topologies for subterminal schemes.

    \subsection{The Full Classification of Subterminal Schemes}

    Finally, we synthesize our findings into a concise classification of subterminal schemes, fully describing both their point structure and topological properties.

    A subterminal scheme is given by the following data:
    \begin{enumerate}
        \item
        An exponent function $e: \P \to \{0, 1, 2, \dots\} \cup \{+\infty\}$.
        \item 
        A binary parameter $q \in \{0,1\}$, where $q = 1$ if there exists a prime $p$ with $e(p) = +\infty$.
        \item 
        An equivalence class $C$ of subsets of primes
        $$
        E := \{p \in \P : e(p) \geq 1\},
        $$
        under the equivalence relation where two subsets are equivalent if and only if their symmetric difference is finite.
    \end{enumerate}

    These data uniquely determine a subterminal scheme, which can be expressed as a colimit of spectra of solid rings, with the following structure:

    \begin{itemize}
        \item 
        The generic point $\Q$, present if and only if $q = 1$.
    
        \item 
        A point for each prime $p \in \P$, where:
        \begin{itemize}
            \item[$\bullet$]
            $e(p) = 0$, corresponds to the zero ring, hence no point.

            \item[$\bullet$]
            $e(p)$ positive integer, corresponds to a points whose stalk is the cyclic ring $\Z/p^{e(p)}\Z$.
    
            \item[$\bullet$]
            $e(p) = +\infty$, corresponds to a points whose stalk is $\Z_{(p)}$.
        \end{itemize}
    \end{itemize}
    \begin{itemize}
        \item 
        Open sets not containing $\Q$ are arbitrary unions the points corresponding to points in
        $$
        E_< := \{p \in \P : e(p) < +\infty\}.
        $$
    
        \item 
        Open sets containing $\Q$ are those whose intersection with $C$ is cofinite in $C$, i.e., such that the intersection of their complementary with $C$ is finite.
    \end{itemize}

    This classification fully characterizes all subterminal schemes.
	\printbibliography

\end{document}